\documentclass[12pt]{article}

\usepackage{amsmath}
\usepackage{amsfonts}
\usepackage{bigints}
\usepackage{amsthm}

\begin{document}

\newcommand{\A}{\mbox{${{{\cal A}}}$}}


\author{Attila Losonczi}
\title{On mean-sets}

\date{10 August 2018}

\newtheorem{thm}{\qquad Theorem}[section]
\newtheorem{prp}[thm]{\qquad Proposition}
\newtheorem{lem}[thm]{\qquad Lemma}
\newtheorem{cor}[thm]{\qquad Corollary}
\newtheorem{rem}[thm]{\qquad Remark}
\newtheorem{ex}[thm]{\qquad Example}
\newtheorem{prb}[thm]{\qquad Problem}
\newtheorem{df}[thm]{\qquad Definition}

\maketitle

\begin{abstract}

\noindent

We introduce a new type of means. It is new in two ways: its domain consists of sets and its values are sets too. We investigate the properties and behavior of such generalization. We also present many naturally arisen examples for such means.

\noindent
\footnotetext{\noindent
AMS (2010) Subject Classifications: 26E60, 40A05, 28A10 \\

}

\end{abstract}

\section{Introduction}
This paper can be considered as a continuation of the investigations started in \cite{lamis} and \cite{lamisii} where we started to build the theory of means on infinite sets. An ordinary mean is for calculating the mean of two (or finitely many) numbers. This can be extended in many ways in order to get a more general concept where we have a mean on some infinite subsets of $\mathbb{R}$. The various general properties of such means, the relations among those means were studied thoroughly in \cite{lamis}, \cite{lamisii},\cite{lamubs},\cite{lambm} and \cite{lamsbm}. 

In this paper we go on in generalization. There are natural examples where the mean of a set $H\subset\mathbb{R}$ is not a single number, instead it is a set which satisfies the generalization of the required mean inequality. We call such set the mean-set of the underlying set $H$. Here we are going to investigate such functions. First we describe many possible properties of mean-sets and some basic relations among them. Some of those properties are generalizations of properties of ordinary/generalized means and some of the new properties can be applied for mean-sets only.
We then study what operations can be defined between mean-sets and also between a mean-set and a generalized mean.

Finally we present many natural examples for mean-sets. We present some that are based on measures and some that are created by using ordinary means. In the last part of the paper we build three mean-sets that are defined by sequences of symmetric sets that approximate the underlying countably infinite set in a natural way. 
\subsection{Basic notations}
For easier readability we copy some basic notations from \cite{lamis} and \cite{lamisii}.

Throughout this paper function $\A()$ will denote the arithmetic mean of any number of variables. Moreover if $(a_n)$ is an infinite sequence and $\lim_{n\to\infty}\A(a_1,\dots,a_n)$ exists then $\A((a_n))$ will denote its limit.

\begin{df}\label{davg}Let $\mu^s$ denote the s-dimensional Hausdorff measure ($0\leq s\leq 1$). If $0<\mu^s(H)<+\infty$ (i.e. $H$ is an $s$-set) and $H$ is $\mu^s$ measurable then $$Avg(H)=\frac{\int\limits_H x\ d\mu^s}{\mu^s(H)}.$$
If $0\leq s\leq 1$ then set $Avg^s=Avg|_{\{\text{measurable s-sets}\}}$. E.g. $Avg^1$ is $Avg$ on all Lebesgue measurable sets with positive measure.
\end{df}

\begin{df}For $K\subset\mathbb{R},\ y\in\mathbb{R}$ let us use the notation $$K^{y-}=K\cap(-\infty,y],K^{y+}=K\cap[y,+\infty).$$
\end{df}

Let $T_s$ denote the reflection to point $s\in\mathbb{R}$ that is $T_s(x)=2s-x\ (x\in\mathbb{R})$. $H\subset\mathbb{R}$ is called symmetric if $\exists s\in\mathbb{R}$ such that $T_s(H)=H$.

If $H\subset\mathbb{R},x\in\mathbb{R}$ then set $H+x=\{h+x:h\in H\}$. Similarly $\alpha H=\{\alpha h:h\in H\}\ (\alpha\in\mathbb{R})$.

We use the convention that these operations $+,\cdot$ have to be applied prior to the set theoretical operations, e.g. $H\cup K\cup L+x=H\cup K\cup (L+x)$.

$int(H), cl(H), H'$ will denote the interior, the closure and accumulation points of $H\subset\mathbb{R}$ respectively. Let $\varliminf H=\inf H',\ \varlimsup H=\sup H'$ for infinite bounded $H$.

If $a,b\in\mathbb{R},a>b$ then we use the convention that $[a,b]=[b,a]$ and similarly for the other types of intervals.

Usually ${\cal{K}},{\cal{M}}$ will denote means, $Dom({\cal{K}})$ denotes the domain of ${\cal{K}}$.

\begin{df}A \textbf{generalized mean} is a function ${\cal{K}}:C\to \mathbb{R}$ where $C\subset P(\mathbb{R})$ consists of some (finite or infinite) bounded subsets of $\mathbb{R}$, $\emptyset\notin C$ and $\inf H\leq {\cal{K}}(H)\leq\sup H$ holds for all $H\in C$. We call ${\cal{K}}$  an \textbf{ordinary mean} if $C$ consists of finite sets only.
\end{df}

\section{Properties of mean-sets}

There are examples where the mean of a set $H\subset\mathbb{R}$ is not a single number, instead it is a set. We call such set the mean-set of $H$. More precisely:

\begin{df}A \textbf{mean-set} is a function ${\cal{MS}}:C\to P(\mathbb{R})$ where $C\subset P(\mathbb{R})$ consists of some (finite or infinite) bounded subsets of $\mathbb{R}$, $\emptyset\notin C$ and ${\cal{MS}}(H)\subset[\inf H,\sup H]$ holds for all $H\in C$.
\end{df}

A mean-set can be considered as a generalization of a generalized mean ${\cal{K}}:C\to \mathbb{R}$  ($C\subset P(\mathbb{R})$) if we take $\tilde{\cal{K}}(H)=\{{\cal{K}}(H)\}$.

\smallskip

Obviously the identity map $id_C(H)=H$ is a mean-set.

\smallskip

In the definition we allow a mean-set having a value as the empty set $\emptyset$ in order to get a fairly general notion. However in most of the cases it will not happen.

\subsection{Basic properties of mean-sets}\label{ssexpprop}
Throughout these subsections ${\cal{MS}}$ will denote a mean-set and $Dom({\cal{MS}})$ will denote its domain. 

\smallskip

Usually we expect $Dom({\cal{MS}})$ to be closed under finite union and intersection and 
${\cal{MS}}(H)$ to be Borel/Lebesgue measurable whenever $H\in Dom({\cal{MS}})$.

\begin{df}\hfill

\begin{itemize}
\item $\cal{MS}$ is called \textbf{internal} if for all $H\in Dom({\cal{MS}})\ {\cal{MS}}(H)\subset[\inf H,\sup H]$. 
${\cal{MS}}$ is \textbf{strong internal} if for all infinite $H\in Dom({\cal{MS}})$ $${\cal{MS}}(H)\subset[\varliminf H,\varlimsup H].$$ 

\item $\cal{MS}$ is \textbf{monotone} if $\sup H_1\leq\inf H_2$ implies that \[\inf{\cal{MS}}(H_1)\leq\inf {\cal{MS}}(H_1\cup H_2)\leq\sup{\cal{MS}}(H_1\cup H_2)\leq\sup {\cal{MS}}(H_2).\] 
${\cal{MS}}$ is \textbf{strong monotone} if ${\cal{MS}}$ is strong internal and $\varlimsup H_1\leq\varliminf H_2$ implies that \[\varliminf{\cal{MS}}(H_1)\leq\varliminf {\cal{MS}}(H_1\cup H_2)\leq\varlimsup{\cal{MS}}(H_1\cup H_2)\leq\varlimsup {\cal{MS}}(H_2).\] 

${\cal{MS}}$ is \textbf{mean-monotone} if $H,K_1,K_2,H\cup K_1,H\cup K_2\in Dom({\cal{MS}})$, \[\sup K_1\leq\inf{\cal{MS}}(H)\leq\sup{\cal{MS}}(H)\leq\inf K_2\text{ implies that }\]
\[\inf{\cal{MS}}(H\cup K_1)\leq\inf{\cal{MS}}(H)\leq\sup{\cal{MS}}(H)\leq\sup{\cal{MS}}(H\cup K_2).\]

\item The mean is \textbf{translation invariant} if $x\in\mathbb{R}, H\in Dom({\cal{MS}})$ then $H+x\in Dom({\cal{MS}}),\ {\cal{MS}}(H+x)={\cal{MS}}(H)+x$.

\item ${\cal{MS}}$ is \textbf{symmetric} if $H\in Dom({\cal{MS}})$ bounded and symmetric implies $T_s({\cal{MS}}(H))={\cal{MS}}(H)$.

${\cal{MS}}$ is \textbf{reflection invariant} if $H\in Dom({\cal{MS}})$ being bounded implies $T_s({\cal{MS}}(H))={\cal{MS}}(T_s(H))$ for all $s\in\mathbb{R}$.

\item $\cal{MS}$ is \textbf{homogeneous} if $H\in Dom({\cal{MS}}),\ \alpha\in\mathbb{R}^+$ then $\alpha H\in Dom({\cal{MS}})$, ${\cal{MS}}(\alpha H)=\alpha {\cal{MS}}(H)$.

\end{itemize}
\end{df}

\begin{prp}\label{pieq}Internality is equivalent to \newline\indent$\inf H\leq\inf{\cal{MS}}(H)$ and $\sup{\cal{MS}}(H)\leq\sup H\ (H\in Dom({\cal{MS}})).$

Strong-internality is equivalent to \newline\indent$\varliminf H\leq\inf{\cal{MS}}(H)$ and $\sup{\cal{MS}}(H)\leq\varlimsup H\ (H\in Dom({\cal{MS}})).$\qed
\end{prp}

\begin{prp}If ${\cal{MS}}$ is strong internal and $H'=\{h\}$ then ${\cal{MS}}(H)=\{h\}$.\qed
\end{prp}

\subsection{Operations}

\begin{prp}Let ${\cal{MS}}_1,{\cal{MS}}_2$ are mean-sets on the same domain. Then $H\mapsto{\cal{MS}}_1(H)\cup{\cal{MS}}_2(H),\ H\mapsto{\cal{MS}}_1(H)\cap{\cal{MS}}_2(H)$ are mean-sets as well. More generally: if $O:P(\mathbb{R})\times P(\mathbb{R})\to P(\mathbb{R})$ is given such that $\min\{\inf A,\inf B\}\leq \inf O(A,B)\leq \sup O(A,B)\leq\max\{\sup A,\sup B\}$ holds then $H\mapsto O({\cal{MS}}_1(H),{\cal{MS}}_2(H))$ is a mean-set too.
\end{prp}
\begin{proof}It is simply because \[\inf H\leq\min\{\inf{\cal{MS}}_1(H),\inf{\cal{MS}}_2(H)\}\leq\inf O({\cal{MS}}_1(H),{\cal{MS}}_2(H)).\] $\sup$ can be handled similarly.
\end{proof}

We can define a natural two variable operation on mean-sets that makes it a semigroup.
\begin{df}Let ${\cal{MS}}_1,{\cal{MS}}_2$ be two mean-sets such that $Ran({\cal{MS}}_2)\subset Dom({\cal{MS}}_1)$. Then $({\cal{MS}}_1\circ{\cal{MS}}_2)(H)={\cal{MS}}_1({\cal{MS}}_2(H))$.
\end{df}

\begin{prp}Let ${\cal{MS}}_1,{\cal{MS}}_2$ be two mean-sets such that $Ran({\cal{MS}}_2)\subset Dom({\cal{MS}}_1)$. Then ${\cal{MS}}_1\circ{\cal{MS}}_2$ is a mean-set. If ${\cal{MS}}_2$ is strong-internal then so is ${\cal{MS}}_1\circ{\cal{MS}}_2$.
Moreover if both mean-sets are translation-invariant, reflection-invariant, homogeneous then so is ${\cal{MS}}_1\circ{\cal{MS}}_2$.
\end{prp}
\begin{proof}By definitions of mean-set 
\[\inf H\leq\inf {\cal{MS}}_2(H)\leq\inf{\cal{MS}}_1({\cal{MS}}_2(H))\] 
whenever $H\in Dom({\cal{MS}})$. $\sup$ can be handled similarly.

Strong-internality similarly follows from \ref{pieq}. The inheritance of the remaining properties is straightforward.
\end{proof}

Clearly the identity map is the unit element in the semigroup.

\begin{df}Let ${\cal{MS}}$ be a mean-set such that $H\in Dom({\cal{MS}})$ implies that ${\cal{MS}}(H)\in Dom({\cal{MS}})$. Then we can define new mean-sets in the following way.
${\cal{MS}}^{(0)}(H)=H,\ {\cal{MS}}^{(n+1)}(H)={\cal{MS}}({\cal{MS}}^{(n)}(H))\ (n\in\mathbb{N}\cup\{0\})$. Clearly ${\cal{MS}}^{(1)}={\cal{MS}}$.

We define one more mean-set: ${\cal{MS}}^{(\infty)}(H)=\bigcup\limits_{i=0}^{\infty}{\cal{MS}}^{(i)}(H)$.
\end{df}

\begin{prp}Let ${\cal{MS}}$ be a mean-set. It has an inverse for operation $\circ$ iff it is injective and $\inf H=\inf {\cal{MS}}(H)$ and $\sup H=\sup {\cal{MS}}(H)$. The mean-sets having these properties constitute a group for this operation.
\end{prp}
\begin{proof}Apply \ref{pieq} to both ${\cal{MS}}(H)$ and ${\cal{MS}}^{-1}(H)$.
\end{proof}

The question arises whether such mean-set exists that is not the identity and defined on a reasonably big subset of $P(\mathbb{R})$. We present such an example.

\begin{ex}Let $Dom\ {\cal{MS}}$ be the set of open bounded subsets of $\mathbb{R}$. If $H$ is an open bounded set, then $H=\bigcup\limits_1^nI_i$ or $H=\bigcup\limits_1^{\infty}I_i$ where $I_i$ is an open interval and $I_i\cap I_j=\emptyset\ (i\ne j)$. Let us remove the middle third closed interval of $I_i$ from $I_i$ and let us denote the resulted set by $I^*_i$ that the union of two open interval. Set ${\cal{MS}}(H)=\bigcup\limits_1^{\infty}I^*_i$. Evidently $\inf H=\inf {\cal{MS}}(H)$ and $\sup H=\sup {\cal{MS}}(H)$. 

We show that ${\cal{MS}}$ is injective. Let $H\in Dom\ {\cal{MS}}$. We have to show that if ${\cal{MS}}(K)={\cal{MS}}(H)$ then $K=H$. We know that ${\cal{MS}}(H)$ is open hence it is a disjoint union of some open intervals. Let $I$ be an interval in this union. By the construction from $H$, $I$  has a pair $J_H$ that has the same length. Similarly it has another pair $J_K$ from the construction from $K$. If $J_H=J_K$ then the same interval exists in both $H$ and $K$. If we have this property for all $I$ then $H=K$ holds. Suppose that $J_H\ne J_K$. Then they are on different sides of $I$ and they came from different intervals from $H$ and $K$. Say $I,J_H$ came from $I^1_H\subset H$, and $I,J_K$ came from $I^1_K\subset K$. But then $J_K$ has to have a pair $I_2$ and an ancestor interval $I^2_H\subset H$. Then $I_2$ has to have a pair $I_3$ by construction from $K$ and so on. We end up with an infinite sequence of intervals with equal length that is a contradiction because $H,K$ were bounded.

Clearly ${\cal{MS}}^{(n)}\ne{\cal{MS}}^{(m)}\ (n\ne m)$ hence ${\cal{MS}}$ generates an infinite cyclic subgroup. \qed
\end{ex}

\begin{df}Product of a mean-sets and a generalized mean: Let ${\cal{MS}},{\cal{K}}$ be a mean-sets and a generalized mean respectively such that $Ran({\cal{MS}})\subset Dom({\cal{K}})$. Then set $({\cal{K}}\circ{\cal{MS}})(H)={\cal{K}}({\cal{MS}}(H))$.
\end{df}

\begin{prp}Let ${\cal{MS}},{\cal{K}}$ be a mean-sets and a generalized mean respectively such that $Ran({\cal{MS}})\subset Dom({\cal{K}})$. Then $({\cal{K}}\circ{\cal{MS}})(H)$ is a generalized mean. If ${\cal{MS}}$ is strong-internal then so is ${\cal{K}}\circ{\cal{MS}}$. Moreover if both ${\cal{MS}},{\cal{K}}$ are translation-invariant, reflection-invariant, homogeneous then so is ${\cal{K}}\circ{\cal{MS}}$.
\end{prp}
\begin{proof}By definitions of mean-set and generalized mean \[\inf H\leq\inf {\cal{MS}}(H)\leq{\cal{K}}({\cal{MS}}(H))\leq\sup {\cal{MS}}(H)\leq\sup H\] whenever $H\in Dom({\cal{MS}})$ .

Strong-internality similarly follows from \ref{pieq}. The inheritance of the remaining properties is straightforward.
\end{proof}

\subsection{Some more properties}

We enumerate some various other properties that we will refer to later. 

\begin{df}\hfill
\begin{itemize}
\item $\cal{MS}$ is \textbf{finite-independent} if $H\in Dom({\cal{MS}})$ is infinite, $V$ is finite then $H\cup V,H-V\in Dom({\cal{MS}})$ and \newline${\cal{MS}}(H)={\cal{MS}}(H\cup V)={\cal{MS}}(H-V).$

\item ${\cal{MS}}$ is \textbf{convex} if $I$ is a closed interval and ${\cal{MS}}(H)\subset I,\ L\subset I,H\cup L\in Dom({\cal{MS}})$ then ${\cal{MS}}(H\cup L)\subset I$. 

\item ${\cal{MS}}$ is called \textbf{closed} if $H,cl(H)\in Dom({\cal{MS}})$ then \newline${\cal{MS}}(cl(H))={\cal{MS}}(H)$.

\item ${\cal{MS}}$ is called \textbf{accumulated} if $H,H'\in Dom({\cal{MS}})$ then \newline${\cal{MS}}(H')={\cal{MS}}(H)$. 

\item ${\cal{MS}}$ is \textbf{finite} if ${\cal{MS}}(H)$ is finite $(H\in Dom({\cal{MS}}))$.

\item ${\cal{MS}}$ is \textbf{idempotent} if ${\cal{MS}}({\cal{MS}}(H))={\cal{MS}}(H)$ $(H,{\cal{MS}}(H)\in Dom({\cal{MS}}))$.

\item ${\cal{MS}}$ is \textbf{increasing} if $H\subset K$ implies that ${\cal{MS}}(H)\subset{\cal{MS}}(K)$.

\item Let $H,K\subset\mathbb{R}$. We set $H\leq K$ if there exists $g:H\to K$ bijection such that $\forall x\in H\ x\leq g(x)$.
${\cal{MS}}$ is \textbf{f-increasing} if $H\leq K$ implies that ${\cal{MS}}(H)\leq{\cal{MS}}(K)$.

\item ${\cal{MS}}$ is \textbf{of finite order for each set} if $\forall H\in Dom({\cal{MS}})\ \exists n\in\mathbb{N}$ such that ${\cal{MS}}^{(n+1)}(H)={\cal{MS}}^{(n)}(H)$.

\item ${\cal{MS}}$ is \textbf{of finite order} if $\exists n\in\mathbb{N}$ such that $\forall H\in Dom({\cal{MS}})\ {\cal{MS}}^{(n+1)}(H)={\cal{MS}}^{(n)}(H)$. Otherwise it is of infinite order.

\end{itemize}
\end{df}

The following statement is a trivial consequence of the definitions.

\begin{prp}Let ${\cal{K}}$ be a generalized mean, ${\cal{MS}}_1,{\cal{MS}}_2$ be mean-sets such that $Ran({\cal{MS}}_2)\subset Dom({\cal{K}})$ and $Ran({\cal{MS}}_2)\subset Dom({\cal{MS}}_1)$. If ${\cal{MS}}_2$ is finite-independent, convex, closed, accumulated then so are ${\cal{K}}\circ{\cal{MS}}_2$ and ${\cal{MS}}_1\circ{\cal{MS}}_2$. If both ${\cal{MS}}_1,{\cal{MS}}_2$ are increasing or f-increasing then so is ${\cal{MS}}_1\circ{\cal{MS}}_2$.\qed
\end{prp}

\section{Examples}

\begin{ex}Many of the usual topological operators are mean-sets. For example $H\mapsto int(H),\ H\mapsto cl(H),\ H\mapsto H',\ H\mapsto int(cl(H))$ are all mean-sets.
\end{ex}

\begin{prp}The mean-set ${\cal{MS}}(H)=int(H)$ is strongly-internal, strong-monotone, mean-monotone, translation invariant, reflection invariant, homogeneous, convex, finite-independent, idempotent, increasing, ${\cal{MS}}(H)\subset H$. It is not closed, not accumulated.
\end{prp}
\begin{proof}Most of the mentioned properties are well-known, here we just show the not-known, not completely trivial ones.

${\cal{MS}}$ is strongly-internal since $H-[\varliminf H,\varlimsup H]$ is countable hence cannot have an interior point.

To show that ${\cal{MS}}$ is strong-monotone we can assume that $\sup H_1\leq\inf H_2$ because we can leave countable many points without changing the result. Similarly in the statement to be proved we can replace $\varliminf$ with $\inf$ and $\varlimsup$ with $\sup$ for which the statement is straightforward.

$H=\mathbb{Q}\cap[0,1]$ shows that ${\cal{MS}}$ is not closed, not accumulated.
\end{proof}

\begin{df}\label{dmsaa}Let ${\cal{MS}}^{aa}(H)=\{\frac{x+y}{2}:x,y\in H\}$ and \newline\indent ${\cal{MS}}^{aas}(H)=\{\frac{x+y}{2}:x,y\in H,x\ne y\}$.
\end{df}

\begin{df}If $z\in {\cal{MS}}^{aa}(H)$ then let 
\[z^-=\{x\in H:2z-x\in H,x\leq z\},\ z^+=\{y\in H:2z-y\in H,z\leq y\}.\]
\end{df}

\begin{ex}${\cal{MS}}^{aa}$ is not of finite order for each set: Let $H=\{\frac{1}{n}:n\in\mathbb{N}\}$. Then ${\cal{MS}}^{(n)}(H)\ne{\cal{MS}}^{(n+1)}(H)\ (n\in\mathbb{N})$.
\end{ex}

\begin{prp}If ${\cal{MS}}={\cal{MS}}^{aa}$ then ${\cal{MS}}^{(\infty)}(H)$ is a dense subset of $[\inf H,\sup H]$.
\end{prp}
\begin{proof}Let $x\in(\inf H,\sup H)$. Then there are $a,b\in H$ such that $a<x<b$. Clearly $a,b,\frac{a+b}{2}\in {\cal{MS}}^{(1)}(H)$. By induction one can readily show that ${\cal{MS}}^{(n)}(H)$ contains a partition of $[a,b]$ finer than $\frac{b-a}{2^n}$ which gives that $x$ is an accumulation point of ${\cal{MS}}^{(\infty)}(H)=\bigcup\limits_{i=0}^{\infty}{\cal{MS}}^{(i)}(H)$.
\end{proof}

${\cal{MS}}^{aas}$ behaves in a different way as the next example shows.

\begin{ex}$({\cal{MS}}^{aas})^{(n)}(\{-1,0,1\})=\{0,\pm\frac{1}{2^n}\}$, hence $({\cal{MS}}^{aas})^{(\infty)}(\{-1,0,1\})=\{0\}\cup\{\pm\frac{1}{2^n}:n\in\mathbb{N}\}$.
\end{ex}

\begin{prp}Let $H\subset\mathbb{R}$ be bounded. Then \[({\cal{MS}}^{aa}(H))'=\Big\{\frac{x+y}{2}:x,y\in H'\ or\ (x\in H',\ y\in H)\Big\}.\]
\end{prp}
\begin{proof}If $x,y\in H'$ or $(x\in H',\ y\in H)$ then clearly $\frac{x+y}{2}\in ({\cal{MS}}^{aa}(H))'$.

Let $p\in({\cal{MS}}^{aa}(H))'$. Then there are sequences $(x_n),(y_n)$ such that $x_n,y_n\in H$ and $\frac{x_n+y_n}{2}\to p$. By boundedness we can assume that that $x_n\to x,y_n\to y$. Evidently both sequence cannot be constant. If none of them is constant then we get the first case ($x,y\in H'$). If one sequence is constant then we get the second case $(x\in H',\ y\in H)$. 
\end{proof}

\begin{ex}We cannot omit boundedness. Let $H=\{-n-\frac{1}{n},n+\frac{2}{n}:n\in\mathbb{N}\}$. Clearly $H'=\emptyset,\ 0\in({\cal{MS}}^{aa}(H))'$.\qed
\end{ex}

\begin{prp}If $x\in H'$ then $x\in({\cal{MS}}^{aas}(H))'$.
\end{prp}
\begin{proof}If $x\in H'$ then there is $(x_n)$ such that $\forall n\ x_n\in H,\ x_n\ne x_m\ (n\ne m)$ and $x_n\to x$ moreover either $\forall n\ x_n<x$ or $\forall n\ x_n>x$ and $|x-x_n|$ is strictly decreasing. Then clearly $x\in({\cal{MS}}^{aas}(\{x_n:n\in\mathbb{N}\}))'$.
\end{proof}

\begin{prp}${\cal{MS}}^{aa}$ and ${\cal{MS}}^{aas}$ is not strong internal, not finite, not strong-monotone, not convex. ${\cal{MS}}^{aa}$ is convex while ${\cal{MS}}^{aas}$ is not. They are both monotone, increasing, translation-invariant, reflection-invariant, homogeneous. $H\subset{\cal{MS}}^{aa}(H)$ holds.
\end{prp}
\begin{proof}The example $H=\{\frac{1}{n}:n\in\mathbb{N}\}\cup\{-\frac{1}{n}:n\in\mathbb{N}\}$ shows that ${\cal{MS}}^{aa}$ is not strong internal, not finite, not strong-monotone (let $H_1=\{\frac{1}{n}:n\in\mathbb{N}\},H_2=\{-\frac{1}{n}:n\in\mathbb{N}\}$). ${\cal{MS}}^{aa}$ is covex since if ${\cal{MS}}^{aa}(H)\subset I$ then $H\subset I$ for an interval $I$. For showing that ${\cal{MS}}^{aas}$ is not convex let $H=\{-1,2\},I=[0,1],L=\{0\}$.

The remaining statements are trivial consequences of the properties of the arithmetic mean. The last assertion $H\subset{\cal{MS}}^{aa}(H)$ is straightforward.
\end{proof}

\begin{prp}If $H$ is Borel measurable then so are ${\cal{MS}}^{aa}(H),{\cal{MS}}^{aas}(H)$ and $z^-,z^+$. Moreover $\lambda(z^-)=\lambda(z^+)$.
\end{prp}
\begin{proof}Let $f:H\times H\to\mathbb{R}, f(x,y)=\frac{x+y}{2}.$ Then $f$ is continuous hence $f(H\times H)={\cal{MS}}^{aa}(H)$ is Borel measurable because $H\times H$ is Borel measurable as well. For ${\cal{MS}}^{aas}(H)$ it is enough to note that ${\cal{MS}}^{aas}(H)=f(H\times H-\{(x,x):x\in\mathbb{R}\})$.

Set $\pi_x(x,y)=x$ the projection to the first coordinate. We then get the statement for $z^-$ by $z^-=\pi_x\big(f^{-1}(\{z\})\cap H\times H\big)\cap(-\infty,z]$. Similarly for $z^+$.

For proving $\lambda(z^-)=\lambda(z^+)$ it is enough to observe that $z^+=2z-z^-=\{2z-x:x\in z^-\}$.
\end{proof}

We can investigate $g(z)=\lambda(z^-)\ (z\in {\cal{MS}}^{aa}(H))$ and ask if it has a maximum and if yes then there is a unique $z$ where it is reached.

First let us fix some notation: $H\triangle K=(H-K)\cup (K-H)\ (H,K\subset\mathbb{R})$, $\pi_x(x,y)=x$ is the projection to the first coordinate, set $f(x,y)=\frac{x+y}{2}$, $H_z=f^{-1}(\{z\})\cap H\times H$ when $H\subset\mathbb{R},\ z\in\mathbb{R}$, and $\lambda_1$ denotes the 1-dimensional Lebesgue measure if we measure a line segment in the plane.

\begin{lem}\label{l2ihzc}Let $I,J$ be two open interval. Then $z\mapsto\lambda_1\big(f^{-1}(\{z\})\cap I\times J\big)$ is continuous.
\end{lem}
\begin{proof}Actually $f^{-1}(\{z\})\cap I\times J$  is the intersection of $I\times J$ and the line through $(z,z)$ with slope $-1$ and its length is clearly continuous function of $z$.
\end{proof}

\begin{lem}\label{llls}Let $H,K\subset\mathbb{R}$ Borel measurable and $d=\lambda(H\triangle K)$. Then $|\lambda_1(H_z)-\lambda_1(K_z)|\leq2\sqrt{2}d$ ($z\in\mathbb{R}$).
\end{lem}
\begin{proof}Evidently 
\[|\lambda_1(H_z)-\lambda_1(K_z)|\leq|\lambda_1\big(f^{-1}(\{z\})\cap (H\times H\triangle K\times K)\big)|\leq\]
\[2\sqrt{2}\lambda(\pi_x(H\times H\triangle K\times K))=2\sqrt{2}d.\qedhere\]
\end{proof}

\begin{lem}\label{lbmawo}Let $H\subset\mathbb{R}$ be bounded Lebesgue measurable, $\epsilon>0$. Then there is $K\subset\mathbb{R}$ such that $K$ is the union of finitely many open intervals and $\lambda(H\triangle K)<\epsilon$.
\end{lem}
\begin{proof}It is known that there is a bounded open $G$ such that $H\subset G$ and $\lambda(G-H)<\frac{\epsilon}{2}$. If $G=\bigcup\limits_{i=1}^{\infty}I_i$ where $I_i$ are disjoint open intervals then let $j\in\mathbb{N}$ be chosen such that $\sum\limits_{i=j+1}^{\infty}\lambda(I_i)<\frac{\epsilon}{2}$. Then $K=\bigcup\limits_{i=1}^{j}I_i$ will suit.
\end{proof}

\begin{thm}\label{tgcont}Let $H\subset\mathbb{R}$ be bounded Borel measurable. Let $g:\mathbb{R}\to\mathbb{R},$
\[g(z)=\begin{cases}
\lambda(z^-)&\text{when }z\in {\cal{MS}}^{aa}(H)\\
0&\text{otherwise}.
\end{cases}\]
Then $g$ is continuous.
\end{thm}
\begin{proof}We already noted that \[\lambda(z^-)=\lambda\big(\pi_x(H_z)\cap(-\infty,z]\big)\] and it is true for every $z\in\mathbb{R}$. From that $\lambda(z^-)=\frac{1}{2\sqrt{2}}\lambda_1\big(H_z\big)$. Hence it is enough to show that $z\mapsto\lambda_1\big(H_z\big)$ is continuous.

Let $\epsilon>0$ and $z_0\in\mathbb{R}$ be given. We need a $\delta>0$ such that 
$|\lambda_1\big(H_{z_0}\big)-\lambda_1\big(H_z\big)|<\epsilon$
whenever $|z_0-z|<\delta$.

By \ref{lbmawo} there is $K\subset\mathbb{R}$ such that $K$ is the union of finitely many open intervals and $\lambda(H\triangle K)<\frac{\epsilon}{6\sqrt{2}}$. Let $K=\bigcup\limits_{i=1}^nI_i,\ I_i\cap I_j=\emptyset\ (i\ne j)$.
Clearly $\lambda_1(K_z)=\sum\limits_{i,j=1}^n\lambda_1\big(f^{-1}(\{z\})\cap I_i\times I_j\big)$ hence $z\mapsto\lambda_1(K_z)$ is continuous because of \ref{l2ihzc}. 

Obviously
\[|\lambda_1(H_{z_0})-\lambda_1(H_z)|<
|\lambda_1(H_{z_0})-\lambda_1(K_{z_0})|+
|\lambda_1(K_{z_0})-\lambda_1(K_z)|+
|\lambda_1(K_z)-\lambda_1(H_z)|.\]
By \ref{llls} $|\lambda_1(H_{z_0})-\lambda_1(K_{z_0})|<\frac{\epsilon}{3}$ and also $|\lambda_1(K_z)-\lambda_1(H_z)|<\frac{\epsilon}{3}$. Now by the continuity of $z\mapsto\lambda_1(K_z)$ we can choose $\delta>0$ such that $|\lambda_1(K_{z_0})-\lambda_1(K_z)|<\frac{\epsilon}{3}$ holds whenever $|z_0-z|<\delta$.
\end{proof}

\begin{cor}Let $H\subset\mathbb{R}$ be bounded Borel measurable. Then $g:{\cal{MS}}^{aa}(H)\to\mathbb{R},\ g(z)=\lambda(z^-)$ takes its maximum value.
\end{cor}
\begin{proof}By \ref{tgcont} the extended $g$ is continuous on $[\inf {\cal{MS}}^{aa}(H),\sup {\cal{MS}}^{aa}(H)]$ hence it takes its maximum in this compact interval. If $z\not\in {\cal{MS}}^{aa}(H)$ then $g(z)=0$ hence the maximum has to be taken in a point of ${\cal{MS}}^{aa}(H)$.
\end{proof}

The following example shows that $z\mapsto\lambda(z^-)$ may take its maximum in several points, moreover we cannot state that the set where its maximum is taken is connected.

\begin{ex}Let $H=[0,2]\cup[5,6]$. Then $Avg(H)=2.5$.

One can easily show that ${\cal{MS}}^{aa}(H)=[0,2]\cup[2.5,4]\cup[5,6]$ and 
\[\lambda(z^-)=
\begin{cases}
z & \text{if }0\leq z\leq 1 \\
2-z & \text{if }1\leq z\leq 2 \\
2z-5 & \text{if }2.5\leq z\leq 3 \\
1 & \text{if }3\leq z\leq 3.5 \\
8-2z & \text{if }3.5\leq z\leq 4 \\
z-5 & \text{if }5\leq z\leq 5.5 \\
6-z & \text{if }5.5\leq z\leq 6 \\
\end{cases},\]
and $\lambda(z^-)=0$ where not specified.

Hence the maximum value (that is 1) is taken in $\{1\}\cup[3,3.5]$.
\qed
\end{ex}

\begin{df}For $0<\lambda(H)<+\infty$ let \[{\cal{MS}}^{hf}(H)=\{x:\lambda(H^{x-})=\lambda(H^{x+})\}.\]
\end{df}

\begin{prp}\label{pmshf1}Let $0<\lambda(H)<+\infty$. Then

(1) ${\cal{MS}}^{hf}(H)\ne\emptyset$

(2) ${\cal{MS}}^{hf}\subset[\varliminf H,\varlimsup H]$ i.e. it is strongly-internal

(3) ${\cal{MS}}^{hf}(H)$ is a finite closed interval

(4) $\lambda\big(H\cap {\cal{MS}}^{hf}(H)\big)=0$

(5) Let $p\in{\cal{MS}}^{hf}(H)$. Then ${\cal{MS}}^{hf}(H)=\{p\}\iff\forall\delta>0\ \lambda\big(H\cap(p-\delta,p)\big)>0$ and $\lambda\big(H\cap(p,p+\delta)\big)>0$.
\end{prp}
\begin{proof}(1) Set $f(x)=\lambda(H^{x-})\ (x\in\mathbb{R})$. Then $\lim\limits_{-\infty}f=0,\ \lim\limits_{+\infty}f=\lambda(H)$ and $f$ is continuous. Therefore there is $x\in\mathbb{R}$ such that $f(x)=\frac{\lambda(H)}{2}$ i.e. $x\in{\cal{MS}}^{hf}(H)$.

\medskip

(2) If $x<\varliminf H$ then $\lambda(H^{x-})=0$ which means that $x\not\in{\cal{MS}}^{hf}(H)$. The other case is similar.

\medskip

(3) Clearly ${\cal{MS}}^{hf}(H)=f^{-1}\big(\frac{\lambda(H)}{2}\big)$ which gives that ${\cal{MS}}^{hf}(H)$ is closed by $f$ being continuous. $f$ being increasing gives that ${\cal{MS}}^{hf}(H)$ has to be connected. ${\cal{MS}}^{hf}(H)$ being infinite interval would contradict to either $\lim\limits_{-\infty}f=0$ or $\lim\limits_{+\infty}f=\lambda(H)$.

\medskip

(4) If $x,y\in\mathbb{R},\ x<y$ then $f(y)-f(x)=\lambda(H\cap(x,y])$. Let now ${\cal{MS}}^{hf}(H)=[a,b]$. Then $0=f(b)-f(a)=\lambda\big(H\cap({\cal{MS}}^{hf}(H)-\{a\})\big)$.

\medskip

(5) To show sufficiency let ${\cal{MS}}^{hf}(H)=[a,b]$. By (3) we get that $a=b=p$. To prove necessity suppose that $\exists\delta>0$ such that $\lambda\big(H\cap(p,p+\delta)\big)=0$. Then $p+\frac{\delta}{2}$ would be in ${\cal{MS}}^{hf}(H)$ too that is a contradiction. The other case is similar.
\end{proof}

\begin{prp}\label{pmshf2}Let ${\cal{MS}}^{hf}(H)=[a,b]$. Then \[Avg({\cal{MS}}^{hf}(H))=Avg(H) \iff Avg(H^{b+})-b=a-Avg(H^{a-}).\]
\end{prp}
\begin{proof}By \ref{pmshf1} (4) 
\[Avg(H)=\frac{\lambda(H^{a-})Avg(H^{a-})+\lambda(H^{b+})Avg(H^{b+})}{\lambda(H)}=\frac{Avg(H^{a-})+Avg(H^{b+})}{2}\]
because $\lambda(H^{a-})=\lambda(H^{b+})=\frac{\lambda(H)}{2}$. Then $Avg({\cal{MS}}^{hf}(H))=\frac{a+b}{2}$ gives the statement.
\end{proof}

\begin{ex}$Avg({\cal{MS}}^{hf}(H))\ne Avg(H)$ in general. Let $H=[0,2]\cup[3,4]\cup[5,6]$. Then ${\cal{MS}}^{hf}(H)=[2,3],\ Avg([0,2])=1,\ Avg([3,4]\cup[5,6])=4.5$ which gives the statement by \ref{pmshf2}.
\end{ex}

\begin{prp}The mean-set ${\cal{MS}}^{hf}$ is strongly-internal, strong-monotone, convex, translation invariant, reflection invariant, homogeneous, finite-independent. It is not closed, not accumulated, not idempotent, not increasing.
\end{prp}
\begin{proof}The "strong" parts and finite-independence are a consequence of the fact that a countable set has measure zero.

To show convexity let ${\cal{MS}}^{hf}(H)=[a,b]\subset I$ where $I$ is an interval and let $L\subset I$. Suppose that $x\in {\cal{MS}}^{hf}(H\cup L)$ and $\sup I<x$ holds as well. Obviously $\lambda(H^{b-})=\lambda(H^{b+})=\lambda(H\cap[b,x])+\lambda(H^{x+})$, and $\lambda(H^{b-})+\lambda(H\cap[b,x])+\lambda(L)=\lambda((H\cup L)^{x-})=\lambda((H\cup L)^{x+})=\lambda(H^{x+})$. From these two equations we get that $\lambda(L)=0$ which means that ${\cal{MS}}^{hf}(H\cup L)={\cal{MS}}^{hf}(H)$ which gives that $x\in [a,b]\subset I$ that is a contradiction.

The remaining assertions are trivial.

To show that it is neither closed nor accumulated let $H=([0,1]\cap\mathbb{Q})\cup[2,3]$. It is not idempotent: $H=[0,1]\cup[2,3]$. To show that it is not increasing let $H_1=[0,1]\cup[4,5],\ H_2=[0,2]\cup[3,5]$. Then $H_1\subset H_2$ but ${\cal{MS}}^{hf}(H_1)\not\subset{\cal{MS}}^{hf}(H_2)$.
\end{proof}

\begin{ex}We construct a mean-set ${\cal{MS}}$ that is of finite order for each set, but is not of finite order.

Let $Dom({\cal{MS}})$ be the set of all bounded open subsets of $\mathbb{R}$. If $H\in Dom({\cal{MS}})$ then $H=\cup_1^{\infty}I_i$ where $I_i$ is an open interval and $I_i\cap I_j=\emptyset\ (i\ne j)$. Because $H$ is bounded $m=\max\{length(I_i):i\in\mathbb{N}\}$ exists and $\{I_i:length(I_i)=m\}$ is finite. Let $\{I_i:length(I_i)=m\}=\{I_{n_1},\dots,I_{n_k}\}$ such that $\sup I_{n_j}<\inf I_{n_{j+1}}\ (1\leq j\leq k-1)$. If $k=1$ then set ${\cal{MS}}(H)=I_{n_1}$. If $k>1$ then set ${\cal{MS}}(H)=\bigcup\limits_{j=2}^k I_{n_j}$. It is easy to see that ${\cal{MS}}$ has the required property.\qed
\end{ex}

\subsection{Mean-sets defined by ordinary means}

We can also define mean-sets starting from ordinary means. In this subsection we present some ways where we meet such mean-sets in a natural way.

\smallskip

First let us set a convention for an ordinary mean ${\cal{K}}$: if $a>b$ then set ${\cal{K}}(a,b)={\cal{K}}(b,a)$ i.e. we could say that ${\cal{K}}$ is defined on unordered pairs of $\mathbb{R}$ or on the set $\{a,b\}$.

\smallskip

We can simply generalize the method described in Definition \ref{dmsaa}.

\begin{df}Let ${\cal{K}}$ be an ordinary mean defined for any pairs of points and $H\subset\mathbb{R}$. Set ${\cal{MS}}^{aa}_{\cal{K}}(H)=\{{\cal{K}}(a,b):a,b\in H\}$. 
\end{df}

Similarly we can prove:

\begin{prp}${\cal{MS}}^{aa}_{\cal{K}}$ is internal, monotone and increasing. It is not strong internal, not finite, not strong-monotone, not convex.\qed
\end{prp}

In the sequel we will define mean-sets that are defined on finite closed intervals only.

\medskip

The next example is based on the method how we get quasi-arithmetic means from the arithmetic mean.

\begin{df}Let ${\cal{K}}$ be an ordinary mean that is defined on $\mathbb{R}\times\mathbb{R}$. Let $f:\mathbb{R}\to\mathbb{R}$ be a continuous function. Then for $a,b\in\mathbb{R}$ set \[{\cal{K}}^f(a,b)=\big\{c\in[a,b]:f(c)={\cal{K}}\big(f(a),f(b)\big)\big\}.\]
\end{df}

\begin{prp}Let ${\cal{K}}$ be an ordinary mean, $a<b\in\mathbb{R}$. Then

(1) ${\cal{K}}^f(a,b)=[a,b]\cap f^{-1}\big({\cal{K}}(f(a),f(b))\big)$

(2) ${\cal{K}}^f(a,b)\ne\emptyset$

(3) If $f^{-1}(\{a\})$ is discrete for all $a\in\mathbb{R}$ then ${\cal{K}}^f(a,b)$ is finite. 
\end{prp}
\begin{proof}(1) and (3) are obvious.
(2) is the consequence of the Darboux property of the continuous $f$.
\end{proof}

E.g. ${\cal{K}}^{\sin x}$ is a finite mean-set for every ordinary mean ${\cal{K}}$.

\medskip

Let us recall some basic definition regarding ordinary means. An ordinary mean ${\cal{K}}$ is strict internal if $a<b$ implies that $a<{\cal{K}}(a,b)<b$. ${\cal{K}}$ is called continuous if it is a 2-variable continuous function.

\begin{df}Let ${\cal{K}}$ be a strictly internal ordinary mean that is defined on $\mathbb{R}\times\mathbb{R}$.  Let $H\subset\mathbb{R}$ be finite, $H=\{h_1,\dots,h_n\},\ h_1<\dots<h_n$. Then let \[\overline{\cal{K}}(H)=\{{\cal{K}}(h_i,h_{i+1}):1\leq i\leq n-1\}.\]
Let $a<b$ be given. Then set $\overline{\cal{K}}^1(a,b)=\{a,b\}\cup{\cal{K}}(a,b),\ \overline{\cal{K}}^{n+1}(a,b)=\overline{\cal{K}}^n(a,b)\cup\overline{\cal{K}}(\overline{\cal{K}}^n(a,b)),\ \overline{\cal{K}}^{\infty}(a,b)=\bigcup\limits_{n=1}^{\infty}\overline{\cal{K}}^n(a,b)$. 
\end{df}

\begin{prp}$\overline{\cal{K}}^n(a,b),\overline{\cal{K}}^{\infty}(a,b)$ are mean-sets $(n\in\mathbb{N})$.\qed
\end{prp}

\begin{prp}If ${\cal{K}}$ is strictly internal and continuous then $\overline{\cal{K}}^{\infty}(a,b)$ is dense in $(a,b)$.
\end{prp}
\begin{proof}Suppose the contrary that there is $[c,d]\subset[a,b]$ such that $(c,d)\cap\overline{\cal{K}}^{\infty}(a,b)=\emptyset$. Take a maximal such interval. Clearly $c,d\not\in\overline{\cal{K}}^{\infty}(a,b)$ but there are sequences $(c_n),(d_n)$ such that $c_n,d_n\in \overline{\cal{K}}^{\infty}(a,b)\ (n\in\mathbb{N})$ and $c_n\to c,\ d_n\to d$. By strict internality $c<{\cal{K}}(c,d)<d$. By continuity there is $N\in\mathbb{N}$ such that $n>N$ implies that $c<{\cal{K}}(c_n,d_n)<d$ holds which is a contradiction.
\end{proof}

In the construction of compounding means we also meet a natural mean-set.

\begin{df}Let ${\cal{K}}_1,{\cal{K}}_2$ be strictly internal, continuous ordinary means such that ${\cal{K}}_1\leq{\cal{K}}_2$. Let $a<b\in\mathbb{R}$. Set $a_0=a,b_0=b$ and $a_{n+1}={\cal{K}}_1(a_n,b_n),\ b_{n+1}={\cal{K}}_2(a_n,b_n)$. Set ${\cal{MS}}_{{\cal{K}}_1,{\cal{K}}_2}(a,b)=\{a_n,b_n:n\in\mathbb{N}\}$.
\end{df}

It is known that $(a_n),(b_n)$ both converges to the same limit that is called the compounding mean of ${\cal{K}}_1,{\cal{K}}_2$. Hence the mean-set is a union of two convergent sequence.
\subsection{Mean-sets by sequences of approximating sets}

In this subsection we are dealing with mean-sets defined only on countably infinite sets.

\begin{df}\label{d2}
Let $H\subset\mathbb{R}$ bounded $,\ |H|=\aleph_0$. 
Let us say that $(H_n)$ is an \textbf{approximating sequence} for $H$ if $H_n\subset\mathbb{R},\ |H_n|<\infty,H_n\subseteq H_{n+1} (\forall n),\cup_{n=1}^{\infty}H_n=H$. Moreover $(H_n)$ is called \textbf{x-balanced} $(x\in\mathbb{R})$ if $|H^{x-}_n|=|H^{x+}_n|\ (\forall n).$ Let 
$(H_n)$ be called \textbf{balanced} if $(H_n)$ is $x$-balanced for some $x\in\mathbb{R}$.

We define 3 mean-sets, the average, s-average and xs-average sets of $H$.
$$\boldsymbol{{\cal{MS}}^{a}}(H)=\Big\{x\in\mathbb{R}: \exists (H_n)\text{ approximating }  H\text{ and } \lim_{n\to\infty}\A(H_n)=x\Big\}$$
$$\boldsymbol{{\cal{MS}}^{as}}(H)=\Big\{x\in\mathbb{R}: \exists (H_n)\ \text{balanced, approximating }  H\text{ and } \lim_{n\to\infty}\A(H_n)=x\Big\}$$
$$\boldsymbol{{\cal{MS}}^{axs}}(H)=\Big\{x\in\mathbb{R}: \exists (H_n)\ x\text{-balanced, approximating }  H\text{ and } \lim_{n\to\infty}\A(H_n)=x\Big\}.$$
\end{df}

Obviously ${\cal{MS}}^{axs}(H)\subset {\cal{MS}}^{as}(H)\subset {\cal{MS}}^a(H)$.

\bigskip
Unfortunately ${\cal{MS}}^a$ is out of interest as the next proposition shows. It works as the Cesaro summation in this sense (cf. \cite{larrttfc}). (We emphasize that $H$ is bounded.)

\begin{prp}\label{phm}${\cal{MS}}^a(H)=[\varliminf H,\varlimsup H]$.
\end{prp}
\begin{proof}If $\alpha\in[\varliminf H,\varlimsup H]$ then by \cite{larrttfc} Theorem 2.6 we can put $H$ into a sequence $(d_n)$ such that $\A(d_1,\dots,d_n)\to \alpha.$  Obviously $H_n=\{d_i:1\leq i\leq n\}$ gives that $\alpha\in {\cal{MS}}^a(H)$.
\end{proof}

\begin{thm}${\cal{MS}}^{as}(H)=\emptyset$ if $H'=\{\alpha\}$ and $|H^{\alpha-}|<\infty$ or $|H^{\alpha+}|<\infty$.

${\cal{MS}}^{as}(H)=\{\alpha\}$ if $H'=\{\alpha\}$ and $|H^{\alpha-}|=|H^{\alpha+}|=\aleph_0$.

In the remaining cases (when $|H'|>1$) ${\cal{MS}}^{as}(H)=[\alpha,\beta]$ where $$\alpha=\frac{\min H'+\inf (H'-\{\min H'\})}{2},\beta=\frac{\max H'+\sup (H'-\{\max H'\})}{2}.$$
\end{thm}
\begin{proof}The first two assertions are evident. 

We prove the last one. Clearly $H'-\{\min H'\}\ne\emptyset, H'-\{\max H'\}\ne\emptyset$.

Set $a_1=\min H',a_2=\inf (H'-\{\min H'\}),a_3=\sup (H'-\{\max H'\}),a_4=\max H'$. With this notation we have to prove that ${\cal{MS}}^{as}(H)=[\frac{a_1+a_2}{2},\frac{a_3+a_4}{2}]$.

If $(H_n)$ is $y$-balanced then $\A(H_n)=\frac{\A(H^{y-}_n)+\A(H^{y+}_n)}{2}$. By Proposition \ref{phm} $\lim\limits_{n\to\infty}\A(H^{y-}_n)\in {\cal{MS}}^{a}(H^{y-})=[\varliminf H^{y-},\varlimsup H^{y-}]$ and similarly $\lim\limits_{n\to\infty}\A(H^{y+}_n)\in {\cal{MS}}^{a}(H^{y+})=[\varliminf H^{y+},\varlimsup H^{y+}]$. Then $\lim_{n\to\infty}\A(H_n)\in\{x\in\mathbb{R}:x=\frac{x_1+x_2}{2},\exists y$ such that $,\ x_1\in[\varliminf H^{y-},\varlimsup H^{y-}],x_2\in[\varliminf H^{y+},\varlimsup H^{y+}]\}$. 

If $a_1\leq y\leq a_2$ then $[\varliminf H^{y-},\varlimsup H^{y-}]=\{a_1\},\ [\varliminf H^{y+},\varlimsup H^{y+}]=[a_2,a_4]$. Then $\lim\limits_{n\to\infty}\A(H_n)\in[\frac{a_1+a_2}{2},\frac{a_1+a_4}{2}]$ if $(H_n)$ is $y$-balanced.

If $a_3\leq y\leq a_4$ then $[\varliminf H^{y-},\varlimsup H^{y-}]=[a_1,a_3],\ [\varliminf H^{y+},\varlimsup H^{y+}]=\{a_4\}$. Then $\lim\limits_{n\to\infty}\A(H_n)\in[\frac{a_1+a_4}{2},\frac{a_3+a_4}{2}]$ if $(H_n)$ is $y$-balanced.

If $a_2\leq y\leq a_3$ then $\lim\limits_{n\to\infty}\A(H_n)\in[\frac{a_1+a_2}{2},\frac{a_3+a_4}{2}]$ if $(H_n)$ is $y$-balanced.

The join of these three intervals is ${\cal{MS}}^{as}(H)=[\frac{a_1+a_2}{2},\frac{a_3+a_4}{2}]$.
\end{proof}

\begin{lem}\label{lmxs}$x\in {\cal{MS}}^{axs}(H)\Longleftrightarrow \exists x_1\in[\varliminf H^{x-},\varlimsup H^{x-}],\ \exists x_2\in[\varliminf H^{x+},\varlimsup H^{x+}]$ such that $x=\frac{x_1+x_2}{2}$.
\end{lem}
\begin{proof}It simply comes from the proof of the previous theorem where we stated: $\lim\limits_{n\to\infty}\A(H_n)\in\{x\in\mathbb{R}:x=\frac{x_1+x_2}{2},\exists y$ such that $,\ x_1\in[\varliminf H^{y-},\varlimsup H^{y-}],x_2\in[\varliminf H^{y+},\varlimsup H^{y+}]\}$ where $(H_n)$ was $y$-balanced. If $x=y$ then we get the statement. 
\end{proof}

\begin{thm}${\cal{MS}}^{axs}(H)=\emptyset$ if $H'=\{\alpha\}$ and $|H^{\alpha-}|<\infty$ or $|H^{\alpha+}|<\infty$.

${\cal{MS}}^{axs}(H)=\{\alpha\}$ if $H'=\{\alpha\}$ and $|H^{\alpha-}|=|H^{\alpha+}|=\aleph_0$.

${\cal{MS}}^{axs}(H)=\{\frac{\alpha+\beta}{2}\}$ if $H'=\{\alpha,\beta\}\ (\alpha\ne\beta)$.

In the remaining cases (when $|H'|>2$) ${\cal{MS}}^{axs}(H)\subseteq[\alpha,\beta]$ where $$\alpha=\frac{\min H'+\inf (H'-\{\min H'\})}{2},\beta=\frac{\max H'+\sup (H'-\{\max H'\})}{2}.$$
Moreover ${\cal{MS}}^{axs}(H)$ is a union of countably many non-degenerative intervals.
\end{thm}
\begin{proof}The first three assertions are evident. ${\cal{MS}}^{axs}\subset {\cal{MS}}^{as}$ gives the fourth.

We prove that ${\cal{MS}}^{axs}(H)$ is a union of countably many non-degenerative intervals. We have to show that if $x\in {\cal{MS}}^{axs}(H)$ then there is a non-degenerative interval $I$ such that $x\in I$ and $I\subset {\cal{MS}}^{axs}(H)$. 

If $x\in[\min H',\max H']$ then $|H'|>2$ gives that one of the intervals $[\varliminf H^{x-},\varlimsup H^{x-}],\ [\varliminf H^{x+},\varlimsup H^{x+}]$ is non-degenerative. By Lemma \ref{lmxs} we know that $x\in {\cal{MS}}^{axs}(H)$ implies that $x\in [\frac{\varliminf H^{x-}+\varliminf H^{x+}}{2},\frac{\varlimsup H^{x-}+\varlimsup H^{x+}}{2}]$. 

There are two cases.

(1) $x\notin H'$. Then $I=[\frac{\varliminf H^{x-}+\varliminf H^{x+}}{2},\frac{\varlimsup H^{x-}+\varlimsup H^{x+}}{2}]\bigcap (\varlimsup H^{x-},\varliminf H^{x+})\subset {\cal{MS}}^{axs}(H)$ where both intervals are non-degenerative and so is $I$ (obviously $x\in I$).

(2) $x\in H'-\{\varliminf H,\varlimsup H\}$. Let assume that $x=\varlimsup H^{x-}$ (the other case can be handled similarly). Then $x=\frac{x_1+x_2}{2}$ where $ x_1\in[\varliminf H^{x-},\varlimsup H^{x-}],x_2\in[\varliminf H^{+x},\varlimsup H^{+x}]$. Let us take the minimum such $x_1$. Then $x_1$ cannot be equal to $x$ because it would imply that $x_2=x$ i.e. $x$ would be a right sided accumulation point as well hence $\varliminf H^{+x}=x$ would hold so we could choose a smaller $x_1$ and a greater $x_2$ too. Therefore $x_1<x=\varlimsup H^{x-}$. That gives that $\exists\epsilon>0$ such that $[x,x+\epsilon)\subset {\cal{MS}}^{axs}(H)$ because $\frac{y_1+x_2}{2}\in{\cal{MS}}^{axs}(H)$ for $y_1\in[x_1,x)$.
\end{proof}

\begin{ex}\hfill
\begin{itemize}

\item $H_1=\{\frac{1}{n}:n\in\mathbb{N}\}\cup\{1+\frac{1}{n}:n\in\mathbb{N}\}$

${\cal{MS}}^{a}(H_1)=[0,1],\ {\cal{MS}}^{as}(H_1)={\cal{MS}}^{axs}(H_1)=\{\frac{1}{2}\}$

\item $H_2=\{\frac{1}{n}:n\in\mathbb{N}\}\cup\{1+\frac{1}{n}+\frac{1}{k}:n,k\in\mathbb{N}\}$.

${\cal{MS}}^a(H_2)=[0,2],\ {\cal{MS}}^{as}(H_2)={\cal{MS}}^{axs}(H_2)=[0.5,1.75]$

\item $H_3=\{\frac{1}{n}:n\in\mathbb{N}\}\cup\{1-\frac{1}{n}:n\in\mathbb{N}\}\cup\{5+\frac{1}{n}:n\in\mathbb{N}\}$.

${\cal{MS}}^a(H_3)=[0,5],{\cal{MS}}^{as}(H_3)=[0.5,3],{\cal{MS}}^{axs}(H_3)=[0.5,1)\cup[2.5,3]$
\end{itemize}
\end{ex}


{\footnotesize

\noindent
Dennis G\'abor College, Hungary 1119 Budapest Fej\'er Lip\'ot u. 70.

\noindent E-mail: losonczi@gdf.hu, alosonczi1@gmail.com\\
}
\end{document}